\documentclass[10pt,a4paper,reqno]{amsart}
\usepackage{amsthm}
\usepackage{amsmath}
\usepackage{amssymb}
\usepackage{graphicx,color}

\usepackage[font=small]{caption}
\usepackage{subcaption}

\usepackage{multirow}
\usepackage[shortlabels]{enumitem}

\usepackage{soul} 

\usepackage{cite}

\makeatletter
\theoremstyle{plain}
\newtheorem{theorem}{Theorem}
\theoremstyle{definition}
\theoremstyle{remark}
\newtheorem{remark}[theorem]{Remark}
\theoremstyle{plain}
\newtheorem{proposition}[theorem]{Proposition}
\theoremstyle{plain}
\newtheorem{lemma}[theorem]{Lemma}
\theoremstyle{plain}

\theoremstyle{definition}

\usepackage{amsfonts}
\usepackage{mathrsfs}

\newtheorem*{question*}{\it{QUESTION}}

\theoremstyle{plain}

\newtheorem*{oq*}{Open Question}

\newcommand{\N}{\mathbb{N}}
\newcommand{\R}{{\mathbb{R}}}
\newcommand{\C}{{\mathbb{C}}}

\newcommand{\dd}{{\rm d}}
\newcommand{\ii}{{\rm i}}
\newcommand{\e}{{\rm e}}

\renewcommand{\Re}{\mathop\mathrm{Re}\nolimits}
\renewcommand{\Im}{\mathop\mathrm{Im}\nolimits}

\newcommand{\dist}{\mathop\mathrm{dist}\nolimits}

\makeatother

\begin{document}

\title[]{On the discrete spectrum of non-selfadjoint operators with applications to Schr\"odinger operators with complex potentials}

\author{Sabine B{\"o}gli}
\address[Sabine B{\"o}gli]{
	Department of Mathematical Sciences, Durham University, Upper Mountjoy, Stockton Road, Durham DH1 3LE, UK
}
\email{sabine.boegli@durham.ac.uk}

\author{Sukrid Petpradittha}
\address[Sukrid Petpradittha]{
	Department of Mathematical Sciences, Durham University, Upper Mountjoy, Stockton Road, Durham DH1 3LE, UK
}
\email{sukrid.petpradittha@durham.ac.uk}

\subjclass[2020]{47A75, 35P15, 47B44, 47A10}

\keywords{Eigenvalue counting estimate, Birman--Schwinger principle, Cwikel--Lieb--Rozenblum inequality, Lieb--Thirring inequality, Schr{\" o}dinger operator, complex potential.}
\date{\today}

\begin{abstract}
For relatively form-compact perturbations of non-negative selfadjoint operators, we obtain an upper bound on the number of discrete eigenvalues in half-planes separated from the positive real axis. The bound is given in terms of a partial trace of the real part of the Birman--Schwinger operator, or an appropriate rotation thereof. While eigenvalue counting estimates of this type are classical in the selfadjoint setting, no analogous connection between  the number of discrete eigenvalues and the Birman--Schwinger operator has previously been established in the non-selfadjoint theory. The proof proceeds via techniques in antisymmetric tensor product spaces that serve as a non-selfadjoint replacement for the classical arguments. As an application to Schr\"odinger operators, we generalise the Cwikel--Lieb--Rozenblum inequality to complex potentials and derive new Lieb--Thirring type inequalities. We also analyse the sharpness of the obtained bounds and discuss their optimality within the considered framework.
\end{abstract}

\maketitle

\section{Introduction}
We prove an abstract eigenvalue counting estimate for a broad class of non-selfadjoint operators.
Let $H_0$ be a non-negative  selfadjoint operator and denote by $H$ the operator obtained by adding a relatively form-compact perturbation.
Assuming for simplicity that the perturbation can be written as an operator $V$ (see Theorem~\ref{thm:B-S} for the general situation when $H$ is induced by quadratic forms), 
we prove that, for  parameters $\alpha\in\R$ and $\varepsilon\geq 0$, the number $N=N(\Re \lambda+ \alpha\Im \lambda<-\varepsilon;H)$ of discrete eigenvalues of $H$  in the half-plane  
$$\{\lambda\in\C:\,\Re \lambda+ \alpha\Im \lambda<-\varepsilon\}$$
(counting algebraic multiplicities) satisfies
\begin{equation}\label{eq:NIntro}
N\leq \sum_{j=1}^N E_j(-\Re S-\alpha \Im S).
\end{equation}
Here $S$ is the Birman--Schwinger operator $S=(H_0+\varepsilon)^{-1/2}V(H_0+\varepsilon)^{-1/2}$, which is compact.
Throughout the paper, for a selfadjoint, bounded from above operator~$T$, we denote by $E_j(T)$ the $j$-th eigenvalue of  $T$, ordered in non-increasing sense and counting multiplicities; if there are fewer than $j$ eigenvalues above the essential spectrum, we set $E_j(T)=\max\sigma_{\rm e}(T)$. 
Note that the right-hand side of~\eqref{eq:NIntro} can be further estimated by the full trace of the negative part, ${\rm Tr}((\Re S+\alpha \Im S)_-)$, but in our application we will use a different further estimate of the partial trace (still depending on $N$) and eventually solve the inequality for $N$.
Here and in the following, $t_- = \max\{-t,0\}$ denotes the negative part of a real number~$t$, and analogously $T_-$ denotes the negative part of a selfadjoint operator $T$.

While the bound in \eqref{eq:NIntro} is interesting in itself and widely applicable, we demonstrate its usefulness by applying it to Schr\"odinger operators $H=-\Delta+V$ in the Hilbert space $L^2(\R^d)$ for any dimension $d\in\N$ where now $V$ is the operator of multiplication with a complex-valued potential. 
Setting $p=d/2+\gamma$ with
\begin{equation}\label{eq:gamma}
\gamma\geq 1/2\quad (d=1), \quad \gamma>0\quad  (d=2), \quad \gamma\geq 0 \quad (d\geq 3),
\end{equation}
we prove (see Theorem~\ref{thm:CLR-complex}) that there exists a constant $C_{d,p}>0$ such that for all $V\in L^{p}(\R^d)$, all $\alpha\in\R$ and $\varepsilon>0$,
\begin{equation}\label{eq:NalphaV}
N(\Re \lambda+ \alpha\Im \lambda<-\varepsilon;-\Delta+V)\leq C_{d,p}\,\varepsilon^{-\gamma} \int_{\R^d} (\Re V(x)+\alpha\Im V(x))_-^{p}\,\dd x.
\end{equation}
For $d\geq 3$, setting $\gamma=0$, $\alpha=0$, taking the limit as $\varepsilon\to 0$ and considering only real-valued potentials reduces to the well-known Cwikel--Lieb--Rozenblum inequality.
Our generalisation to the non-selfadjoint case closes a gap in the existing theory of non-selfadjoint Schr\"odinger operators. We also use it to prove new Lieb--Thirring type inequalities (see Theorem~\ref{thm:LT complex-B}).

In the following, we first review the Cwikel--Lieb--Rozenblum inequality and the related Lieb–Thirring inequalities, emphasising the contrast between the selfadjoint and non-selfadjoint settings, and then describe our proof methods.

\subsection*{Background on the Cwikel--Lieb--Rozenblum (CLR) inequality}
The CLR inequality states that for dimension $d\geq 3$ there exists a constant $C_d>0$ such that for any real-valued potential $V\in L^{d/2}(\R^d)$  the number $N=N(\lambda<0;-\Delta+V)$ of negative eigenvalues satisfies 
$$N\leq C_d \int_{\mathbb{R}^d} (V(x))_-^{d/2}\,\mathrm{d}x.$$
The CLR inequality was proved independently by  Cwikel, Lieb and Rozenblum in the 1970s. The proofs of Cwikel \cite{cwikel1977weak} and Lieb \cite{Lieb1976CLR,Lieb1980CLR} both use the Birman–Schwinger principle but they differ in the way they estimate the partial trace: Lieb’s approach relies on interpolation and Sobolev-type estimates, whereas Cwikel’s method is based on operator ideals and weak Schatten-class bounds and is adaptable to a wider class of operators. Rozenblum’s proof \cite{Rozenblum1972CLR,Rozenblum1976CLR} is substantially different, relying on spectral and variational techniques rather than the Birman–Schwinger framework.

In the non-selfadjoint setting, the discrete eigenvalues exhibit substantially more complicated behaviour: they are typically non-real and may, in principle, accumulate at non-zero points of the essential spectrum $[0,\infty)$. This phenomenon was first demonstrated by Pavlov \cite{Pavlov2}, who constructed a Schr\"odinger operator in $L^2(0,\infty)$
with a rapidly decaying real potential and infinitely many discrete eigenvalues.
In his construction,  the non-selfadjointness arises solely from a complex Robin boundary condition $f'(0)=z f(0)$ (for a $z\in\C$). 
In a similar spirit, but with the non-selfadjointness coming from the potential itself, the first author constructed a Schr\"odinger operator with a complex potential in $L^p(\R^d)$ (in any dimension $d$) such that infinitely many discrete eigenvalues accumulate \emph{everywhere} at the essential spectrum (see \cite{bogli2017schrodinger} for $p>d$ and the joint work with Cuenin \cite{bogli2023counterexample} for the relaxed condition $p>(d+1)/2$).
In fact, the CLR inequality does not hold in the non-selfadjoint case; the number $N(-\Delta+V)$ of \emph{all} discrete eigenvalues is not controlled by $\|V\|_{L^{d/2}(\R^d)}^{d/2}$: the present authors (jointly with \v{S}tampach \cite{bogli2025lieb})  studied the purely imaginary potential $V_h=\ii h\chi_{B_1(0)}$ in $d\geq 2$, with $h>0$ and $\chi_{B_1(0)}$ denoting the characteristic function of the unit ball in $\R^d$. For  $p=d/2$ one can estimate
$$ \frac{N(-\Delta+V_h)}{\|V_h\|_{L^p(\R^d)}^p}=\frac{1}{\|V_h\|_{L^p(\R^d)}^p}\sum_{\lambda\in\sigma_{\rm d}(-\Delta+V_h)} 1 
\geq \frac{1}{\|V_h\|_{L^p(\R^d)}^p}\sum_{\lambda\in\sigma_{\rm d}(-\Delta+V_h)} \frac{{\rm dist}(\lambda,[0,\infty))^{p}}{|\lambda|^{d/2}}$$
and it was shown that the right-hand side diverges logarithmically in $h$ in the strong-coupling limit as $h\to \infty$.
This proves that for complex potentials no CLR inequality can hold for \emph{all} discrete eigenvalues.

Pavlov \cite{Pavlov1} proved in dimensions $d=1$ and $d=3$ that the total number of eigenvalues $N(-\Delta+V)$ is finite if 
$$\sup_{x\in\R^d} \exp(-\delta|x|^{1/2})|V(x)|<\infty$$
for some $\delta>0$, but the result provided no quantitative bound on the number $N(-\Delta+V)$ in terms of the potential.
In~\cite{FrankLaptevSafronov} Frank, Laptev and Safronov proved a quantitative  bound in odd dimensions: in $L^2(0,\infty)$ with Dirichlet boundary condition $f(0)=0$, the bound is
$$N(-\Delta+V)\leq \inf_{\delta>0}\,\frac{1}{\delta^2}\left(\int_0^\infty \exp(\delta x)|V(x)|\,\mathrm d x\right)^2,$$
and in odd dimensions $d\geq 3$ there exists a constant $C_d>0$ such that
$$N(-\Delta+V)\leq \inf_{\delta>0}\,\frac{C_d}{\delta^2}\left(\int_0^\infty \exp(\delta |x|)|V(x)|^{(d+1)/2}\,\mathrm d x\right)^2.$$

Our result in Eq.~\eqref{eq:NalphaV} (see Theorem~\ref{thm:CLR-complex}) establishes a quantitative CLR bound for the non-selfadjoint case which holds in complex half-planes that do not intersect the essential spectrum.

\subsection*{Background on Lieb--Thirring (LT) inequalities}

The LT inequalities are named after Lieb and Thirring \cite{lie-thi_prl75,LTdetailed} and play a central role in mathematical physics, notably in the proof of stability of matter, see also \cite{frank2022schrodinger} for an overview.
For $p=d/2+\gamma$ with $\gamma$ satisfying \eqref{eq:gamma},
the LT inequalities state that there exists a constant $C_{d,p}>0$ such that for any real-valued $V\in L^p(\R^d)$,
$$\sum_{\lambda\in\sigma_{\rm d}(-\Delta+V)} \lambda_-^\gamma \leq C_{d,p}\int_{\mathbb{R}^d} (V(x))_-^{p}\,\mathrm d x.$$
Note that the CLR inequality corresponds to $\gamma=0$.

For  $\gamma\geq 1$, Frank, Laptev, Lieb, and Seiringer \cite{frank2006lieb} proved that the LT inequalities remain valid for Schr\"odinger operators with complex potentials, provided that we sum only over the eigenvalues in appropriate regions of the complex plane. Specifically, the bounds hold for eigenvalues in the left half-plane or, more generally, outside the sector  $\{\lambda\in\C\,:\,|\Im \lambda |<\kappa \Re \lambda\}$ around the positive real axis (for $\kappa>0$). 
In this case, there exists a constant $C_{d,p}>0$ such that for all $V\in L^p(\R^d)$ and $\kappa>0$,
\begin{equation}\label{eq:FLLS}
	\sum_{\lambda\in\sigma_{\rm d}(-\Delta+V),\, |\Im\lambda|\geq\kappa\Re\lambda} |\lambda|^{\gamma}
	\leq C_{d,p}\left(1+\frac{2}{\kappa}\right)^{p}\int_{\R^d} |V(x)|^{p}\;\dd x.
\end{equation}
Note that the $\kappa$-dependent constant on the right-hand side blows up as $\kappa\to 0$.
Moreover, the estimate fails if all discrete eigenvalues are taken into account; explicit counterexamples are provided in \cite{bogli2017schrodinger,bogli2023counterexample,bogli2021lieb, bogli2025lieb}.

Using Eq.~\eqref{eq:NalphaV}, in Lemma~\ref{lem:eigenvalue outside cones} we prove that Eq.~\eqref{eq:FLLS} continues to hold respectively for $\gamma>1/2$ (if $d=1$), $\gamma>0$ (if $d=2$), $\gamma\geq 0$ (if $d\geq 3$),
and in Theorem~\ref{thm:LT complex-B} we derive new LT type inequalities where the left-hand side also depends on ${\rm dist}(\lambda,[0,\infty))$.
This builds on earlier research to find analogues of the LT inequalities for the complex-potential case, see 
e.g.\ \cite{bogli2017schrodinger,boegli2023improved,bogli2023counterexample,bogli2025lieb,bogli2021lieb,
cuenin2022schrodinger,
demuth2013eigenvalues,fra_18,frank2006lieb,GolStep,han_11,
demuth2009discrete,BPoptimal}.  

\subsection*{Why a new proof method was needed for \boldmath $\gamma<1$}
The proof method in \cite{frank2006lieb} does not rely on any structure specific to Schr\"odinger operators and was subsequently extended in \cite{BruneauOuhabaz}, again  for $\gamma\geq 1$, to 
more general non-selfadjoint operators $H$ for which its real part $\Re H$ (induced by the real part of the quadratic form) is bounded from below. 
The argument for eigenvalues in the left half-plane proceeds as follows.
One first shows that any $N$ discrete eigenvalues in the left half-plane, ordered such that $\Re \lambda_1 \leq \Re \lambda_2\leq \Re \lambda_3\leq  \cdots\leq \Re\lambda_N$, satisfy
$$\forall\, 1\leq k\leq N:\quad \sum_{j=1}^k (\Re \lambda_j)_- \leq \sum_{j=1}^k E_j( -\Re H) \leq \sum_{j=1}^k E_j((\Re H)_-).$$
For matrices an inequality of this type goes back to Fan \cite[Thm.~2]{KyFan1950}. 
Given such an inequality between two ordered sets of non-negative numbers $\{x_j\}_{j=1}^N$, $\{y_j\}_{j=1}^N$ (here $x_j=(\Re \lambda_j)_-$, $y_j= E_j((\Re H)_-)$), a classical result of Hardy, Littlewood and P\'olya \cite{Hardy1929convex} implies that for any convex function $\Phi:[0,\infty)\to [0,\infty)$ with $\Phi(0)=0$,
$$\forall\, 1\leq k\leq N:\quad \sum_{j=1}^k \Phi(x_j)\leq \sum_{j=1}^k \Phi(y_j).$$
This estimate was rediscovered by Aizenman and Lieb \cite{aizenman1978semi} for the convex function
 $x\mapsto x^\gamma$ with $\gamma\geq 1$ (note that it is not convex for $\gamma<1$), which yields
\begin{equation}\label{eq:gamma<1}
\sum_{j=1}^N (\Re \lambda_j)_-^{\gamma} \leq \sum_{j=1}^N E_j((\Re H)_-)^{\gamma}.
\end{equation}
Now, if  $H=-\Delta+V$ is a non-selfadjoint Schr\"odinger operator, the selfadjoint LT inequalities can be applied to the negative eigenvalues of $\Re H=-\Delta+\Re V$ to obtain the half-plane bounds in  \cite{frank2006lieb}. Eigenvalue estimates in rotated half-planes, and hence outside a sector, follow by comparing $\Re(\e^{\ii \varphi}\lambda_j)$ with the eigenvalues of $\Re(\e^{\ii\varphi}H)$.

The case $0\leq \gamma< 1$ was left open in \cite{frank2006lieb}.
The above proof strategy breaks down in this regime due to the failure of convexity.
In fact, Bruneau and Ouhabaz \cite[p.~6]{BruneauOuhabaz} constructed a matrix counterexample to \eqref{eq:gamma<1}, which remains a counterexample even if the right-hand side is multiplied by an arbitrary constant.
This obstruction shows that a fundamentally different proof method is required for $\gamma<1$.

Our proof method relies on the Birman–Schwinger principle. For Schr\"odinger operators, this is a well-established tool: it states that $\lambda\in\C\backslash [0,\infty)$ is an eigenvalue of $-\Delta+V$ if and only if $-1$ is an eigenvalue of a corresponding compact Birman--Schwinger operator, which can be written formally as $|V|^{1/2}{\rm sgn}(V)(-\Delta-\lambda)^{-1}|V|^{1/2}$. For $\lambda<0$ it is often more convenient to work with $(-\Delta-\lambda)^{-1/2}V(-\Delta-\lambda)^{-1/2}$. For a general treatment of the Birman--Schwinger principle in the non-selfadjoint setting, we refer to \cite{generalizedB-S}.
For non-selfadjoint Schrödinger operators, Frank \cite{Frank1} used the Birman--Schwinger principle to derive bounds on individual eigenvalues. In the selfadjoint case, one has the stronger result (see \cite[Thm.~4.24]{frank2022schrodinger}) stating that the number $N$ of eigenvalues $\lambda_{j}$ of $-\Delta+V$ with $\lambda_j<-\varepsilon$ is equal to the number of eigenvalues $\mu_j<-1$ of the selfadjoint operator $S=(-\Delta+\varepsilon)^{-1/2}V(-\Delta+\varepsilon)^{-1/2}$ which is compact under the assumptions $V\in L^{d/2+\gamma}(\R^d)$ with $\gamma$ satisfying \eqref{eq:gamma}. In particular, this implies
$$N=\sum_{j=1}^N 1 \leq  \sum_{j=1}^N (\mu_j)_-=\sum_{j=1}^N E_j(-S).$$
Our abstract result in Theorem \ref{thm:B-S} provides a non-selfadjoint analogue of this estimate in the general setting of relatively form-compact perturbations of non-negative selfadjoint operators. This extension is nontrivial, since the variational principle is no longer available in the non-selfadjoint case. Instead, we work in antisymmetric tensor product spaces (see Subsection~\ref{sec:antisymm}), following the general philosophy of \cite{frank2006lieb}, but applying it at the level of the Birman--Schwinger operator rather than directly to the Schr\"odinger operator.
To obtain concrete bounds from our abstract result, one may combine Theorem~\ref{thm:B-S} with the weak Schatten-class estimates due to Cwikel \cite{cwikel1977weak}, as we demonstrate for Schr\"odinger operators.

\section{Abstract eigenvalue counting estimate}

Let $H_{0}$ be a selfadjoint, non-negative operator in a Hilbert space $\mathcal{H}$.
For a (possibly different)  Hilbert space~$\mathcal G$,  assume that $W_0:\mathcal D(W_0)\subset\mathcal H\to \mathcal G$ and $W:\mathcal D(W)\subset\mathcal H\to \mathcal G$ are linear operators such that $\mathcal{D}(H_{0}^{1/2})\subset\mathcal{D}(W)\cap\mathcal{D}(W_{0})$ and $W(H_{0}+I)^{-1/2}$, $W_{0}(H_{0}+I)^{-1/2}$ are compact. Due to \cite[Lem.~B.1]{fra_18}, the quadratic form
\[
	\|H_{0}^{1/2}u\|^{2}+\langle{W_{0}u,Wu}\rangle
\]
with form domain $\mathcal{D}(H_{0}^{1/2})$ is closed, sectorial and induces a unique $m$-sectorial operator $H$ in $\mathcal{H}$. 
Besides, according to \cite[Prop.~B.2]{fra_18}, one obtains the equality of essential spectra $\sigma_{\rm e}(H)=\sigma_{\rm e}(H_{0})$, 
and the discrete spectrum $\sigma_{\rm d}(H):=\sigma(H)\backslash\sigma_{\rm e}(H)$ is at most countable and consists of eigenvalues of finite algebraic multiplicities.

\begin{remark}
The above compactness assumptions imply that the quadratic form $q[u]=\langle{W_{0}u,Wu}\rangle$ 
with form domain $\mathcal D(q)=\mathcal{D}(W)\cap\mathcal{D}(W_{0})\supset\mathcal{D}(H_{0}^{1/2})$
is \emph{relatively form-compact} with respect to the quadratic form $\|H_{0}^{1/2}u\|^{2}$ on $\mathcal{D}(H_{0}^{1/2})$ (see \cite[p.~369]{reedsimonmethods4}).
 In addition, for every $\varepsilon>0$, the operators $W(H_{0}+\varepsilon)^{-1/2}$, $W_{0}(H_{0}+\varepsilon)^{-1/2}$ and hence $(W(H_{0}+\varepsilon)^{-1/2})^{*}(W_{0}(H_{0}+\varepsilon)^{-1/2})$ are compact.
\end{remark}


The following result is the main result of this section. 

\begin{theorem}\label{thm:B-S}
Let $\varepsilon>0$ and $\alpha\in\R$. 
Then, for any $N\in\N$ such that there are $\lambda_{1},\dots,\lambda_{N}\in\sigma_{\rm d}(H)$ (repeated according to their algebraic multiplicities) with $\Re\lambda_{j} + \alpha\Im\lambda_{j}<-\varepsilon$ for all $j=1,\dots,N$, we have
\[
	N\leq \sum_{j=1}^{N}E_j(-\Re S-\alpha\Im S),
\]
where $S:=(W(H_{0}+\varepsilon)^{-1/2})^{*}(W_{0}(H_{0}+\varepsilon)^{-1/2})$ is a compact operator on $\mathcal H$.
In particular, if $(\Re S+\alpha \Im S)_-$ is a trace-class operator, then 
$$N(\Re\lambda+\alpha\Im\alpha<-\varepsilon;H)\leq {\rm Tr}((\Re S+\alpha \Im S)_-)<\infty.$$
\end{theorem}

The proof involves two steps. First we prove a result, using antisymmetric tensor products, which applies only to the geometric eigenspace. Then we demonstrate a perturbation method to break Jordan chains into geometric eigenspaces. Combining these two steps will prove the theorem.

\subsection{Operators acting on antisymmetric tensor products}\label{sec:antisymm}

In the following we will use the notions of antisymmetric tensor products and operators acting on them. We give basic definitions and properties here, and refer for more details to e.g.\ \cite[Sect.~1.5]{simon2005trace} and \cite[Sect.~II.4 and~VIII.10]{reed1980methods}.

Let $N\in\N$. The $N$-fold tensor product $\otimes^N\mathcal{H}:=\mathcal{H}\otimes\cdots\otimes\mathcal{H}$ ($N$ copies)
is equipped with the scalar product defined on simple tensors by
$$\langle \psi_{1}\otimes\cdots\otimes\psi_{N},\phi_{1}\otimes\cdots\otimes\phi_{N}\rangle
:=\prod_{j=1}^{N} \langle{\psi_{j},\phi_{j}}\rangle,$$
and continued to $\otimes^N\mathcal{H}$ by linearity and completion.
Denote by  $\text{perm}(N)$ the group of permutations on $\{1,\dots,N\}$. For a permutation $\pi\in\text{perm}(N)$ let ${\rm sgn}(\pi)\in\{\pm 1\}$ denote its sign, which is the determinant of its corresponding permutation matrix. 
Define the antisymmetrisation operator $\mathcal A_N:\otimes^N\mathcal{H}\to \otimes^N\mathcal{H}$ by
$$\mathcal A_N (\phi_{1}\otimes\cdots\otimes\phi_{N}):=\frac{1}{N!}\sum_{\pi\in\textrm{perm}(N)}{\rm sgn}(\pi)(\phi_{\pi(1)}\otimes\cdots\otimes\phi_{\pi(N)}).$$
Then $\mathcal A_N$ is selfadjoint and $\mathcal A_N^2=\mathcal A_N$, i.e.\ it is an orthogonal projection. 
Its range 
$$\wedge^N\mathcal H:={\rm ran}(\mathcal A_N)$$
is called the \emph{antisymmetric $N$-fold tensor product} of $\mathcal H$. 
Define the \emph{antisymmetric tensor product}
$$\phi_{1}\wedge\cdots\wedge\phi_{N}:=\frac{1}{\sqrt{N!}}\sum_{\pi\in\textrm{perm}(N)}{\rm sgn}(\pi)(\phi_{\pi(1)}\otimes\cdots\otimes\phi_{\pi(N)})
=\sqrt{N!}\,\mathcal A_N (\phi_{1}\otimes\cdots\otimes\phi_{N}),$$
which changes its sign under exchanging of two elements $\phi_i, \phi_j$ for $i\neq j$.
Note that if $\phi_i=\phi_j$ for some $i\neq j$, then $\phi_{1}\wedge\cdots\wedge\phi_{N}=0$.
If $\mathcal H$ is a function space over the variable $x$ (e.g.\ $L^2(\R^d)$), the antisymmetric tensor product $\phi_{1}\wedge\cdots\wedge\phi_{N}$ is a function of $(x_1,\dots,x_N)$ which can be identified with the 
\emph{Slater determinant}
\[
(\phi_{1}\wedge\cdots\wedge\phi_{N})(x_1,\dots,x_N)
=
\frac{1}{\sqrt{N!}}\det((\phi_{i}(x_{j}))_{i,j=1}^N).
\]

A linear operator $T$ in $\mathcal H$ can be lifted to a linear operator $T^{(N)}$ in $\otimes^N\mathcal{H}$ via
$$T^{(N)}(\phi_{1}\otimes\cdots\otimes\phi_{N}):=\sum_{j=1}^N (\phi_{1}\otimes\cdots \otimes T \phi_j\otimes\cdots\otimes\phi_{N}).$$
Then $T^{(N)}$ commutes with $\mathcal A_N$, hence  the antisymmetric space $\wedge^N\mathcal H$ is an invariant subspace under $T^{(N)}$; in the following we only use the operator restricted to $\wedge^N\mathcal H$ and continue denoting it by $T^{(N)}$.

We collect a few properties that we will use afterwards.

\begin{lemma}\label{lem:antisymmprop}
Let $T$ be a linear operator in $\mathcal H$.
For $N$ linearly independent elements $\phi_1,\dots,\phi_N\in\mathcal D(T)$ let 
$\Psi=\phi_{1}\wedge\cdots\wedge\phi_{N}\in\wedge^N\mathcal H$.
Consider the $N\times N$ matrices $A$ and $M$ with entries
\[
A_{ij}=\langle \phi_j, \phi_i\rangle\quad\text{and}\quad M_{ij}=\langle T\phi_j,\phi_i\rangle.
\] 
Then the following properties hold.
\begin{enumerate}
\item[\rm i)] We have $\|\Psi\|^2=\det A>0$.
\item[\rm ii)] If $T$ is selfadjoint and non-negative (resp. uniformly positive), then so is the matrix~$M$; in particular, $A$ is selfadjoint and uniformly positive  and hence $A^{-1/2}$ is well-defined.
\item[\rm iii)] We have $\langle T^{(N)}\Psi,\Psi\rangle=\|\Psi\|^2\, {\rm Tr}(A^{-1/2}MA^{-1/2})$.
\item[\rm iv)] If $T$ is selfadjoint and non-negative, then $\langle T^{(N)}\Psi,\Psi\rangle\geq 0$.
\item[\rm v)] If $T$ is selfadjoint and bounded from above, then 
$$\frac{\langle T^{(N)}\Psi,\Psi\rangle}{\|\Psi\|^2}\leq \sum_{j=1}^N E_j(T).$$
\end{enumerate}
\end{lemma}

\begin{proof}
Due to (2.69)-(2.71) in \cite[Sect.~2.2.2]{friedrich2006theoretical},
$\|\Psi\|^2=\det A$
and, if $\det A \neq 0$, 
\[
\langle T^{(N)}\Psi, \Psi \rangle
=\|\Psi\|^2 \sum_{i,j=1}^N M_{ij}(A^{-1})_{ji}=\|\Psi\|^2 \,{\rm Tr}(MA^{-1});
\]
note that $\det A\neq 0$ due to the linear independence of the $\phi_i$.
Below, in the proof of the claim ii), we show that $A$ is uniformly positive.
Then the cyclic property for traces yields that for any square matrices $A_1,A_2,A_3$ of equal size we have ${\rm Tr}(A_1 A_2 A_3)={\rm Tr}(A_2 A_3 A_1)$. Therefore, we obtain
$${\rm Tr}(MA^{-1}) = {\rm Tr}(A^{-1/2} M A^{-1/2}).$$
This proves claims i) and iii).

Next we prove the claim ii). To this end, assume that $T$ is selfadjoint and non-negative, hence $T^{1/2}$ is well-defined. 
Then $M$ is clearly a selfadjoint matrix. To show non-negativity,  let $c=(c_i)_{i=1}^N\in\C^N$.
Then
$$\langle M c,c\rangle=\sum_{i,j=1}^N M_{ij}c_j\overline{c_i}=\sum_{i,j=1}^N \langle T^{1/2} \phi_j, T^{1/2}\phi_i\rangle c_j\overline{c_i}=\left\|\sum_{i=1}^{N} c_i T^{1/2}\phi_i\right\|^2\geq 0.$$
If $T$ is even uniformly positive, then $\left\|\sum_{i=1}^{N} c_i  T^{1/2}\phi_i\right\|^2=0$ if and only if $c=0$ (due to the linear independence of the $\phi_i$ and hence of the $T^{1/2}\phi_i$).
This proves the claim ii); in particular, if we set $T=I$, then we obtain the claim about $A$.

The claim iv) is an immediate consequence of claims ii) and iii). 
The claim v) follows from the variational principle since the right-hand side of the inequality is equal to the maximum of the spectrum of the selfadjoint operator $T^{(N)}$ in the antisymmetric space $\wedge^N\mathcal H$, see \cite[Thm.~2~(i)]{secondquantization}.
\end{proof}

\begin{remark}
If $T$ is selfadjoint, bounded from above, and we choose $\phi_1,\dots,\phi_N\in\mathcal D(T)$ to be orthonormal, then
$A=I$. Hence Lemma~\ref{lem:antisymmprop} implies that $\|\Psi\|=1$ and
$$\sum_{j=1}^N E_j(T)\geq \langle T^{(N)}\Psi,\Psi\rangle={\rm Tr}(M)=\sum_{j=1}^N \langle T \phi_j,\phi_j\rangle.$$
Note that if $N$ does not exceed the number of discrete eigenvalues of $T$ above the essential spectrum, then we have equality if the $\phi_j$ are chosen to be the eigenvectors corresponding to $E_j(T)$. 
Thus we obtain (see \cite[Prop.~1.33]{frank2022schrodinger})
$$\sum_{j=1}^N E_j(T)=\max\left\{ \sum_{j=1}^N \langle T\phi_j,\phi_j\rangle:\, \phi_1,\dots,\phi_N\text{ orthonormal}\right\}.$$
This was first proved for finite-dimensional spaces by Fan \cite[Thm.~1]{KyFan1949}.
However, in the following we will work with eigenfunctions of non-selfadjoint operators, which are not orthonormal in general, hence we have to work with the general formula in Lemma~\ref{lem:antisymmprop}~iii).
\end{remark}

Now we shall state and prove the first core ingredient for the proof of Theorem~\ref{thm:B-S}. 
Note that, in contrast to Theorem~\ref{thm:B-S}, here we only take the \emph{geometric} multiplicities into account.

\begin{lemma}\label{lem:geometric bd}
Let $K:\mathcal{H}\to\mathcal{H}$ be a bounded linear operator,
and let $\varepsilon>0$ and $\alpha\in\R$.
Then, for any $N\in\N$ such that there are $\lambda_{1},\dots,\lambda_{N}\in\sigma_{\rm d}(H+K)$ (repeated according to their geometric multiplicities) with $\Re\lambda_{j} + \alpha\Im\lambda_{j}<-\varepsilon$ for all $j=1,\dots,N$, we have
\begin{align*}
	N\leq \sum_{j=1}^{N}E_j\left(-\Re S_K-\alpha\Im S_K\right),
\end{align*}
where $S_K:=(W(H_{0}+\varepsilon)^{-1/2})^{*}(W_{0}(H_{0}+\varepsilon)^{-1/2})+(H_{0}+\varepsilon)^{-1/2}K(H_{0}+\varepsilon)^{-1/2}$ is a bounded operator on $\mathcal H$.
\end{lemma}

\begin{remark}\label{rem:closure}
Before we give the proof, let us remark that $S_K$ is the closure of the operator $\tilde S_K-I$
where 
\begin{align*}
	\tilde S_K&:=(H_{0}+\varepsilon)^{-1/2}(H+K+\varepsilon)(H_{0}+\varepsilon)^{-1/2},\\
	 \mathcal D(\tilde S_K)&:=\{\phi\in \mathcal H:\,(H_{0}+\varepsilon)^{-1/2}\phi\in \mathcal D(H)\}.
\end{align*}
Indeed, for $\phi=(H_0+\varepsilon)^{1/2}f\in\mathcal D(\tilde S_K)$ with $f\in\mathcal D(H)$ and $\psi=(H_0+\varepsilon)^{1/2}g\in\mathcal{H}$ with $g\in\mathcal{D}((H_{0}+\varepsilon)^{1/2})$ we have
\begin{align*}
	\langle \tilde S_K \phi,\psi\rangle=\langle (H+K+\varepsilon)f,g\rangle&=
	\langle{H_{0}^{1/2}f,H_{0}^{1/2}g}\rangle+\langle (\varepsilon+K)f,g\rangle+\langle W_0f,Wg\rangle \\
	&=\langle (I+S_K)\phi,\psi\rangle,
\end{align*}
and hence the claim follows due to the density of $\mathcal D(\tilde S_K)\subset\mathcal H$.
Here we have used that,
by the first and second  representation theorems \cite[Thms.~VI.2.1, 2.23]{kato1976perturbation},
$\mathcal D(H)\subset \mathcal D(H_0^{1/2})=\mathcal D((H_0+\varepsilon)^{1/2})$ and $\mathcal D(H)$ is a core of $(H_0+\varepsilon)^{1/2}$.
\end{remark}

\begin{proof}[Proof of Lemma~\ref{lem:geometric bd}]

Consider a finite collection of eigenvalues  $\lambda_{1},\dots,\lambda_{N}$ (repeated according to their {geometric} multiplicities) of $H+K$, such that $\Re\lambda_j + \alpha\Im\lambda_{j} <-\varepsilon$ for all $j=1,\dots,N$.  Then there exists a linearly independent set $\{u_{j}:\,j=1,\dots,N\}$ where $u_{j}$ is an eigenvector corresponding to the eigenvalue $\lambda_j$.  

Let $\phi_{j}:=(H_{0}+\varepsilon)^{1/2}u_{j}$, $j=1,\dots,N$. Due to the first representation theorem \cite[Thm.~VI.2.1]{kato1976perturbation}, $\mathcal{D}(H)\subset\mathcal{D}(H_{0}^{1/2})$, hence the $\phi_{j}$ are well-defined and belong to $\mathcal{D}(\tilde S_K)$. Since the $u_{j}$, $j=1,\dots,N$, are linearly independent, so are the $\phi_{j}$, $j=1,\dots,N$. 
Let 
$\Psi=\phi_{1}\wedge\cdots\wedge\phi_{N}\in\wedge^N\mathcal H$.
 Then $\Psi\not\equiv 0$. In addition, a simple calculation using $(H+K+\varepsilon)u_j=(\lambda_j+\varepsilon)u_j$ shows that
 $$ \tilde S_K^{(N)} \Psi= 
	\left(\sum_{j=1}^{N} (\lambda_{j}+\varepsilon)\right) 
	R^{(N)} \Psi,$$
where we have set $R:=(H_{0}+\varepsilon)^{-1}$ and lifted it to the linear operator $R^{(N)}$ in 
$\wedge^N\mathcal H$. Hence
\begin{equation}\label{eq:crucial}
	\langle{ \tilde S_K^{(N)} \Psi, \Psi }\rangle = 
	\left(\sum_{j=1}^{N} (\lambda_{j}+\varepsilon)\right) 
	\langle{ R^{(N)} \Psi, \Psi }\rangle.
\end{equation}
Since the operator $R=(H_{0}+\varepsilon)^{-1}$ is non-negative, Lemma~\ref{lem:antisymmprop}~iv) implies that
$\langle{ R^{(N)} \Psi, \Psi }\rangle\geq  0$.
Now taking $(\Re+\alpha\Im)$ on both sides of \eqref{eq:crucial}, and using that $\Re\lambda_{j} + \alpha\Im\lambda_{j}<-\varepsilon$ for $j=1,\dots,N$, results in
\begin{equation}\label{eq:ReIm}
	\Re\langle{ \tilde S_K^{(N)} \Psi, \Psi }\rangle+
	\alpha \Im\langle{ \tilde S_K^{(N)} \Psi, \Psi }\rangle \leq 0.
\end{equation}
By Remark~\ref{rem:closure}, we have $\tilde S_K=S_K+I$ on $\mathcal D(\tilde S_K)$.
This implies $\tilde S_K^{(N)}\Psi=S_K^{(N)}\Psi+N\Psi$, hence
$$\langle S_K^{(N)}\Psi,\Psi\rangle=\langle \tilde S_K^{(N)}\Psi,\Psi\rangle-N\|\Psi\|^2.$$
Therefore, setting $T:=-\Re S_K-\alpha\Im S_K$ and using \eqref{eq:ReIm} implies
$$\langle T^{(N)}\Psi,\Psi\rangle
=-\Re \langle S_K^{(N)}\Psi,\Psi\rangle-\alpha \Im \langle S_K^{(N)}\Psi,\Psi\rangle
\geq N \|\Psi\|^2.$$
Now the claim follows from Lemma~\ref{lem:antisymmprop}~v).
\end{proof}

\subsection{Perturbation result to break up Jordan chains}

Let $T$ be a linear operator in $\mathcal{H}$.
Consider a finite collection of eigenvalues $\lambda_{1},\dots,\lambda_{N}\in\sigma_{\rm d}(T)$, with each $\lambda_{j}$  repeated according to its algebraic multiplicity. Then one can find an $N$-dimensional invariant subspace $\mathcal V\subset\mathcal{H}$ under $T$ corresponding to $\lambda_{1},\dots,\lambda_{N}$.
We take a basis of $\mathcal V$ according to the Jordan normal form representing $T$ on the subspace $\mathcal V$. Suppose that the matrix representation has $k\leq N$ Jordan blocks, labelled $n = 1, \dots, k$, one for each distinct eigenvalue $\lambda_{n}$, whose algebraic multiplicity is denoted by $m_{n}$. Note that $\sum_{n=1}^{k} m_{n}=N$. 
For each $n$, let $f_{n,1}, \dots, f_{n,m_{n}}$ be a Jordan chain of length $m_{n}$, i.e.\
\begin{align*}
	(T - \lambda_{n}) f_{n,1} &= 0,\\
	(T - \lambda_{n}) f_{n,\ell} &= f_{n ,\ell-1},\quad \ell=2,\dots, m_{n}.
\end{align*}
Now we define a bounded finite-rank operator $K_0: \mathcal H \to \mathcal H$ 
that acts independently on each Jordan chain as
\begin{align*}
	K_0 f_{n,1} &= f_{n,m_{n}},\\
	K_0 f_{n,\ell} &= 0, \quad  \ell = 2, \dots, m_{n},
\end{align*}
and on the orthogonal complement of $\mathcal V$ we set $K_0|_{\mathcal V^{\perp}}=0$. 

Before we come to the next main ingredient in our proof, we recall that an eigenvalue of a linear operator is called \emph{semisimple} if its algebraic and geometric multiplicities coincide. 

Now we perturb $T$ in such a way that leaves $\mathcal V$ invariant and renders all eigenvalues semisimple.

\begin{lemma}\label{lem:semisimple}
For any $\delta> 0$, all eigenvalues of $T + \delta K_0$ on $\mathcal V$ are semisimple and given by
\[
	z_{n,\ell} = \lambda_{n} -((-1)^{m_{n}} \delta)^{1/m_{n}} {\rm e}^{2 \pi {\rm i} \ell / m_{n}}
\]
for $ \ell = 1, \dots, m_{n}$ and $ n=1,\dots,k$.
\end{lemma} 

\begin{proof}
For $n =1,\dots,k$, the $m_{n} \times m_{n}$ matrix representation of $T|_{\mathcal V}$ in the $n$-th block  with respect to the basis $\{f_{n,1}, \dots, f_{n ,m_{n}}\}$ is given by $\lambda_{n} I + N_{n}$ where
\[
	N_{n} :=
	\begin{pmatrix}
		0 & 1 & 0 & \dots & 0 \\
		0 & 0 & 1 & \ddots & \vdots \\
		0 & 0 & 0 & \ddots & 0 \\
		\vdots & & & \ddots & 1 \\
		0 & 0 & 0 & \dots & 0
	\end{pmatrix}.
\]
Since  $K_0$ acts independently on each chain, the matrix representation of $K_0|_{\mathcal V}$ is block-diagonal with the $n$-th block $M_{n}$ given by the $m_{n}\times m_{n}$ matrix
\[
	M_{n}:=
	\begin{pmatrix}
		0 & 0 & \dots&  & 0 \\
		0 & 0 &  &  & 0 \\
		\vdots & & \ddots & & \vdots \\
		0 & 0 &  &  & 0 \\
		1 & 0 &\dots  &  & 0 
	\end{pmatrix}.
\]
Therefore, the matrix representation of $T + \delta K_0$ on $\mathcal V$ is also block-diagonal. Further, it can be seen that the characteristic polynomial of the $n$-th block is
\[
	\det(\lambda_{n} I+ N_{n} + \delta M_{n} - z I) = (\lambda_{n}-z)^{m_{n}} - (-1)^{m_{n}}\delta,
\]
and its zeros are exactly the claimed eigenvalues $z_{n, \ell}$. This implies the claim.
\end{proof}

\subsection{Proof of the abstract eigenvalue counting estimate}

Now we combine Lemmas~\ref{lem:geometric bd} and~\ref{lem:semisimple} to prove the main result of this section.
To this end, we first show that Lemma~\ref{lem:geometric bd} continues to hold for every finite family of eigenvalues of $H+K$ repeated according to their \emph{algebraic} multiplicities. 

\begin{proposition}\label{prop:algebraic bd}
	Let $K:\mathcal{H}\to\mathcal{H}$ be a bounded linear operator, and let $\varepsilon>0$ and $\alpha\in\R$.
	Then, for any $N\in\N$ such that there exist $\lambda_{1},\dots,\lambda_{N}\in\sigma_{\rm d}(H+K)$ (repeated according to their algebraic multiplicities) with $\Re\lambda_{j} + \alpha\Im\lambda_{j}<-\varepsilon$ for all $j=1,\dots,N$,
	\begin{align*}
		N\leq \sum_{j=1}^{N}E_j\left(-\Re S_K-\alpha\Im S_K\right),
	\end{align*}
	where $S_K:=(W(H_{0}+\varepsilon)^{-1/2})^{*}(W_{0}(H_{0}+\varepsilon)^{-1/2})+(H_{0}+\varepsilon)^{-1/2}K(H_{0}+\varepsilon)^{-1/2}$ is a bounded operator on $\mathcal H$.
\end{proposition}

\begin{proof}
	Let $\varepsilon>0$, $\alpha\in\R$ and  let $\lambda_{1},\dots,\lambda_{N}$ be eigenvalues of $H+K$ such that
	\[
		\Re\lambda_{j} + \alpha\Im\lambda_{j}<-\varepsilon
	\]
	for all $j=1,\dots,N$. Then, in view of Lemma~\ref{lem:semisimple} applied to $T=H+K$, for any $\delta>0$ sufficiently small we obtain a finite family of $N$ simple eigenvalues $z_{j}$ of the perturbed operator $H+K+\delta K_0$ such that $\Re z_{j} + \alpha\Im z_{j}<-\varepsilon$.
	
	Lemma~\ref{lem:geometric bd} with $K$ replaced by $K+\delta K_0$ implies
	\[
		N\leq \sum_{j=1}^{N}E_j\left(-\Re S_{K+\delta K_0}-\alpha\Im S_{K+\delta K_0}\right).
	\]
	Since $K_0$  and $(H_0+\varepsilon)^{-1/2}$ are bounded operators, we have the operator norm convergence
	\[
		\Re S_{K+\delta K_{0}} + \alpha\Im S_{K+\delta K_{0}}\to  \Re S_{K}+\alpha\Im S_{K}
	\]
	as $\delta\to 0$. Therefore, by \cite[Sect.~IV.3.5]{kato1976perturbation}, each eigenvalue $E_j(-\Re {S_{K}}-\alpha\Im {S_{K}})$ is the limit of a sequence of $E_j\left(-\Re S_{K+\delta K_{0}}-\alpha\Im S_{K+\delta K_{0}}\right)$ as $\delta\to 0$. This proves the assertion.
\end{proof}

Now everything is in place to prove the main result.

\begin{proof}[Proof of Theorem~\ref{thm:B-S}]
	The estimate follows by setting $K=0$ in Proposition~\ref{prop:algebraic bd}.
\end{proof}

\section{Applications to non-selfadjoint Schr\"{o}dinger operators}

In this section, we present applications of Theorem~\ref{thm:B-S} to Schr\"{o}dinger operators with complex potentials  in the Hilbert space $L^{2}(\R^{d})$.
In Theorem~\ref{thm:CLR-complex} we generalise the CLR inequality to complex potentials. We, then, prove new LT inequalities for sums over eigenvalues in half-planes in Theorem~\ref{thm:real part LT est} and for sums over \emph{all} discrete eigenvalues in Theorem~\ref{thm:LT complex-B}. Since eigenvalues of Schr\"{o}dinger operators with complex potentials may accumulate at any point in the positive real axis, Theorem~\ref{thm:LT complex-B} provides quantitative information on the accumulation rate, see Remark~\ref{rem:accumulation rate}. 

As in \eqref{eq:gamma}, we assume that 
$$\gamma\geq 1/2\quad (d=1), \quad \gamma>0\quad  (d=2), \quad \gamma\geq 0 \quad (d\geq 3).$$
Then, by \cite[Lem.~4.2]{fra_18}, for any function $W\in L^{d+2\gamma}(\R^{d})$ and any $\varepsilon>0$ the operator $W(-\Delta+\varepsilon)^{-1/2}$ is compact.
Let $V\in L^{d/2+\gamma}(\R^d)$ be complex-valued. Then the assumptions of the previous section are satisfied if we set $\mathcal{H}=\mathcal G=L^{2}(\R^{d}),\,H_{0}=-\Delta,\,W=\sqrt{|V|}$ and $W_{0}=\sqrt{V}$ where $\sqrt{V(x)}:=V(x)/\sqrt{|V(x)|}$ if $V(x)\neq 0$ and $\sqrt{V(x)}:=0$ if $V(x)=0$. The generated $m$-sectorial operator $H=-\Delta+V$ is the Schr\"{o}dinger operator with  potential $V$. Its essential spectrum is $\sigma_{\rm e}(H)=\sigma_{\rm e}(H_0)=[0,\infty)$.

\subsection{CLR inequalities for complex potentials}

In the following result we derive bounds on $N(\Re\lambda+\alpha\Im\lambda<-\varepsilon; -\Delta+V)$ in terms of $\varepsilon$ and the $L^{\gamma+d/2}(\R^d)$-norm of $(\Re V+\alpha\Im V)_-$. In particular, the case $d\geq 3$
and  $\gamma=0$  is the generalisation of the CLR inequality to complex potentials.
The inequality for general $d$ and $\gamma$  is the generalisation of e.g.~\cite[Eq.~(2.8)]{LTdetailed}.

\begin{theorem}\label{thm:CLR-complex}
	Let $p=d/2+\gamma$ with $\gamma$ satisfying \eqref{eq:gamma}.
	Then there exists a constant $C_{d,p}>0$ such that for all $V\in L^{p}(\R^d)$, $\alpha\in\R$ and $\varepsilon>0$,
	$$N(\Re\lambda+\alpha\Im\lambda<-\varepsilon; -\Delta+V)\leq C_{d,p}\,\varepsilon^{-\gamma} \int_{\R^d} (\Re V(x)+\alpha\Im V(x))_-^{p}\,\dd x.$$
	In particular, for $d\geq 3$ and $\gamma=0$ we have, for all $V\in L^{d/2}(\R^d)$ and $\alpha\in\R$,
	$$N(\Re\lambda+\alpha\Im\lambda<0; -\Delta+V)\leq C_{d,d/2}\,\int_{\R^d} (\Re V(x)+\alpha\Im V(x))_-^{d/2}\,\dd x.$$
The constants can be taken as
$$C_{d,p}=
\begin{cases}
1/2 &\text{if } d=1, \,p=1,\\
2^{p+\gamma-4}\pi^{-2p} \tau_d\, d^{d/2} \gamma^{\gamma} p^{p+1} (p-1)^{1-2p} &\text{if } p>\max\{1,d/2\},\\
2^{-5} \pi^{-d} \tau_d\, d^{d+1}(d/2-1)^{1-d} &\text{if } d\geq 3, \,p=d/2,
\end{cases}$$
where $\tau_d$ denotes the volume of the unit ball in $\R^d$.

\end{theorem}

\begin{proof}
	Let $\varepsilon>0$, $\alpha\in\R$ and $V\in L^{p}(\R^d)$. Let $N\in\N$ be such that there are $\lambda_1,\dots,\lambda_N\in\sigma_{\rm d}(-\Delta+V)$ (repeated according to their algebraic multiplicities) with $\Re\lambda_j+\alpha\Im\lambda_{j}<-\varepsilon$ for all $j=1,\dots,N$. 
	Due to Theorem~\ref{thm:B-S}, with $S=(\sqrt{|V|}(-\Delta+\varepsilon)^{-1/2})^{*}(\sqrt{V}(-\Delta+\varepsilon)^{-1/2})$,
	\[
	N\leq\sum_{j=1}^{N} E_{j}(-\Re S-\alpha\Im S).
	\]
	By \cite[Prop~1.33]{frank2022schrodinger},
	$$\sum_{j=1}^{N} E_{j}(-\Re S-\alpha\Im S)
	=\max_{\{e_{j}\}_{j}} \,\sum_{j=1}^{N}\langle{(-\Re S-\alpha\Im S)e_{j},e_{j}}\rangle,$$
	where the maximum is taken over any finite family of orthonormal sequences $\{e_{j}\}_{1\leq j\leq N}\subset\mathcal{D}(S)=\mathcal{H}$;
	if $E_{j}(-\Re S-\alpha\Im S)$ terminates at $j_{0}<N$, i.e. $E_{j}(-\Re S-\alpha\Im S)=0$ for $j=j_{0}+1,\dots,N$, then the maximum is taken over any finite orthonormal family $\{e_{j}\}_{1\leq j\leq j_{0}}$ instead.
	Note that
	\begin{align*}
		\langle{Se_{j},e_{j}}\rangle&=\langle{\sqrt{V}(-\Delta+\varepsilon)^{-1/2}e_{j},\sqrt{|V|}(-\Delta+\varepsilon)^{-1/2}e_{j}}\rangle \\
		&=\int_{\R^d} V(x)\left|((-\Delta+\varepsilon)^{-1/2}e_{j})(x)\right|^2\,{\rm d}x.
	\end{align*}
	Similarly, one has
	\[
	\langle{S^{*}e_{j},e_{j}}\rangle=\langle{e_{j},S e_{j}}\rangle=
	\int_{\R^d} \overline{V(x)}\left|((-\Delta+\varepsilon)^{-1/2}e_{j})(x)\right|^2\,{\rm d}x.
	\]
	Hence, for each $j=1,\dots,N$,
	\begin{align*}
		&\langle{(-\Re S-\alpha\Im S)e_{j},e_{j}}\rangle \\
		&=\int_{\R^d} (-\Re V(x)-\alpha\Im V(x))\left|((-\Delta+\varepsilon)^{-1/2}e_{j})(x)\right|^2\,{\rm d}x \\
		&\leq \int_{\R^d} (\Re V(x)+\alpha\Im V(x))_- \left|((-\Delta+\varepsilon)^{-1/2}e_{j})(x)\right|^2\,{\rm d}x \\
		&=\|(\Re V+\alpha\Im V)^{1/2}_{-}(-\Delta+\varepsilon)^{-1/2}e_{j}\|^2 \\
		&=\langle{(-\Delta+\varepsilon)^{-1/2}(\Re V+\alpha\Im V)_{-}(-\Delta+\varepsilon)^{-1/2}e_{j},e_{j}}\rangle.
	\end{align*}
	Applying \cite[Prop~1.33]{frank2022schrodinger} once more yields
	\begin{equation}\label{eq:N-singular Schrodinger}
		\begin{aligned}
			N&\leq \sum_{j=1}^{N} E_j\left((-\Delta+\varepsilon)^{-1/2}(\Re V+\alpha\Im V)_{-}(-\Delta+\varepsilon)^{-1/2}\right)\\
			&= \sum_{j=1}^{N} s_j ((\Re V+\alpha\Im V)_-^{1/2}(-\Delta+\varepsilon)^{-1/2})^2,
		\end{aligned}
	\end{equation}
	where $s_j(T)$ denotes the $j$-th singular value of an operator $T$, arranged in non-increasing order and repeated according to its multiplicity; we use the convention that if there are fewer than $j$ positive singular values of $T$, we set $s_j (T)=0$.
	
	In the following we use the results of  Cwikel \cite{cwikel1977weak} on weak Lebesgue spaces and weak Schatten class operators. We recall that, for $1\leq q<\infty$, a measurable function $u$ belongs to the weak Lebesgue space $L^{q,\infty}(\R^d)$ if the quasi-norm
	\begin{equation}\label{eq:defweak}
		\|u\|_{L^{q,\infty}(\R^d)}:=\sup_{t>0}\left(t^q |\{x\in\R^d:\,|u(x)|>t\}|\right)^{1/q}
	\end{equation}
	is finite. Analogously, the weak Schatten class $\mathfrak S_{q,\infty}$ consists of all compact operators $A$ in the Hilbert space $L^2(\R^d)$ for which the weak Schatten quasi-norm 
	$$\|A\|_{\mathfrak S_{q,\infty}}:=\sup_{t>0}\left(t^q |\{j\in\N:\,|s_j(A)|>t\}|\right)^{1/q}=\sup_{j\in\N}\left(j^{1/q}s_j(A)\right)$$ 
	is finite; by \cite[p.~9]{simon2005trace} applied to $a_j=s_j (A)^2$ and $p=q/2$, for $2<q<\infty$ the weak Schatten quasi-norm is equivalent to
	$$\sup_{k\in\N}\left(k^{1/q}\left(\frac{1}{k}\sum_{j=1}^k s_j (A)^2\right)^{1/2}\right).$$
	Define the integral operator $B_{u,g}$ in $L^2(\R^d)$ by
	$$(B_{u,g} f)(\xi):=u(\xi)\int_{\R^d} \exp(2\pi\ii \xi\cdot x)g(x)f(x)\,{\rm d}x, \quad \xi\in\R^d.$$
	Then  Cwikel \cite[Sect.~3]{cwikel1977weak} proved that, for $2<q<\infty$, the bilinear map $(u,g)\mapsto B_{u,g}$ is a bounded operator from $L^{q,\infty}(\R^d)\times L^q(\R^d)$ into $\mathfrak S_{q,\infty}$, and
	\begin{equation}\label{eq:Cwikel}
		\sup_{k\in\N}\left(k^{1/q}\left(\frac{1}{k}\sum_{j=1}^k s_j (B_{u,g})^2\right)^{1/2}\right)\leq K_q\, \|u\|_{L^{q,\infty}(\R^d)}\,\|g\|_{L^q(\R^d)}
	\end{equation}
	with the $q$-dependent constant
	$$K_q=\frac{q}{2} \left(\frac{4}{q/2-1}\right)^{1-2/q}\left(1+\frac{2}{q-2}\right)^{1/q}.$$
	
	Let $\mathcal{F}$ denote  the unitary Fourier transform on $L^{2}(\R^{d})$. Following Cwikel \cite{cwikel1977weak}, one can write $(\Re V+\alpha\Im V)_-^{1/2}(-\Delta+\varepsilon)^{-1/2}$ in terms of $B_{u,g}$ with a specific choice of $(u,g)$, namely,
	\begin{align*}
		(\Re V+\alpha\Im V)_-^{1/2}(-\Delta+\varepsilon)^{-1/2}
		&=((-\Delta+\varepsilon)^{-1/2}(\Re V+\alpha\Im V)_-^{1/2})^{*} \\
		&=(\mathcal{F}^{-1}B_{u,g})^{*}=B^{*}_{u,g}\mathcal{F}
	\end{align*}
	with 
	$$u(\xi):=\frac{1}{2\pi}(|\xi|^2+\varepsilon)^{-1/2}, \quad g(x)=(\Re V(x)+\alpha\Im V(x))_-^{1/2}.$$
	This yields
	\[
	s_j((\Re V+\alpha\Im V)_-^{1/2}(-\Delta+\varepsilon)^{-1/2})=s_{j}(B_{u,g}^*)=s_{j}(B_{u,g}).
	\]
	Now \eqref{eq:Cwikel} implies for $k\in\N$ and $2<q<\infty$ that
	$$k^{1/q}\left(\frac{1}{k}\sum_{j=1}^k s_j ((\Re V+\alpha\Im V)_-^{1/2}(-\Delta+\varepsilon)^{-1/2})^2\right)^{1/2}\leq K_q\, \|u\|_{L^{q,\infty}(\R^d)}\,\|g\|_{L^q(\R^d)}.$$
	Hence
	$$\sum_{j=1}^k s_j ((\Re V+\alpha\Im V)_-^{1/2}(-\Delta+\varepsilon)^{-1/2})^2
	\leq k^{1-2/q} \left(K_q \|u\|_{L^{q,\infty}(\R^d)}\,\|g\|_{L^q(\R^d)}\right)^2.$$
	Taking \eqref{eq:N-singular Schrodinger} into account and solving for  $k=N$ gives
	\begin{equation}\label{eq:Cwikelapplication}
		N \leq \left(K_q\, \|u\|_{L^{q,\infty}(\R^d)}\,\|g\|_{L^q(\R^d)}\right)^q
	\end{equation}
	where
	$$\|g\|_{L^q(\R^d)}^q=\int_{\R^d}(\Re V(x)+\alpha\Im V(x))_-^{q/2}\,\dd x.$$
	We want to apply the Cwikel estimate with $q=2p$ but the requirement $2<q<\infty$ means that this is allowed only for $p>1$; thus the case $d=1$, $p=1$ will be treated separately below.
	
	First we consider $p>1$ and set $q=2p$. To determine $\|u\|_{L^{q,\infty}(\R^d)}$ in \eqref{eq:defweak}, we shall note that $|u(\xi)|=\frac{1}{2\pi}(|\xi|^2+\varepsilon)^{-1/2}>t$ if and only if $|\xi|<((2\pi t)^{-2}-\varepsilon)^{1/2}$, which makes sense only for $t<\varepsilon^{-1/2}/(2\pi)$.
	Recall that  $\tau_d$ denotes the volume of the unit ball in $\R^d$.
	For $0<t<\varepsilon^{-1/2}/(2\pi)$ we define the function
	$$F(t):=t^{q}\,|\{\xi\in\R^d:\,|u(\xi)|>t\}|=\tau_d\, t^{q}((2\pi t)^{-2}-\varepsilon)^{d/2}.$$
	For $q>d$ the function attains its maximum
	at $t=\sqrt{q-d}/(2\pi  \sqrt{q \varepsilon})$, which yields
	\begin{align*}
		\|u\|_{L^{q,\infty}(\R^d)}
		&=\left(\sup_{0<t<\varepsilon^{-1/2}/(2\pi)} F(t)\right)^{1/q}
		=\frac{\tau_d^{1/q}}{2\pi} \,\sqrt{\frac{q-d}{q \varepsilon}}\, \left(\frac{d \varepsilon}{q-d}\right)^{d/(2q)} \\
		&=\frac{\tau_d^{1/q}}{2\pi} \frac{(q-d)^{1/2-d/(2q)}d^{d/(2q)}}{\sqrt{q}}\, \varepsilon^{-1/2+d/(2q)}.
	\end{align*}
	Note that $\varepsilon^{-1/2+d/(2q)}=\varepsilon^{-(q-d)/(2q)}$.
	For $q=d$ we have $F(t)=\tau_d\,(2\pi)^{-d}\, (1-(2\pi t)^2 \varepsilon)^{d/2}$, which attains its supremum in the limit $t\to 0$, hence
	$$\|u\|_{L^{d,\infty}(\R^d)}=\left(\sup_{0<t<\varepsilon^{-1/2}/(2\pi)} F(t)\right)^{1/d}=\tau_d^{1/d}\,(2\pi)^{-1}.$$
	Note that this is independent of $\varepsilon$ so the estimate also holds in the limit as $\varepsilon\to 0$.
	Now for $p>1$ the claim follows using $q=2p=d+2\gamma$. 
	
	It remains to consider the case $d=1$, $p=1$. We have
	\begin{align*}
		N  &\leq  \sum_{j=1}^{N} s_j ((\Re V+\alpha\Im V)_-^{1/2}(-\Delta+\varepsilon)^{-1/2})^2 \\
		&\leq \sum_{j=1}^{\infty} s_j ((\Re V+\alpha\Im V)_-^{1/2}(-\Delta+\varepsilon)^{-1/2})^2,
	\end{align*}
	where we have set $s_j ((\Re V+\alpha\Im V)_-^{1/2}(-\Delta+\varepsilon)^{-1/2})=0$ in the case that $j$ is larger than the number of all positive singular values.
	Now we apply the Kato--Seiler-Simon estimate \cite[Thm.~4.1]{simon2005trace} which states that
	for $2\leq q<\infty$ and $a,b\in L^{q}(\R^d)$ one has
	$$\sum_{j=1}^\infty s_j(a(x)b(-\ii\nabla))^{q}
	\leq (2\pi)^{-d}\,\int_{\R^d}|a(x)|^{q}\,\dd x\,\int_{\R^d}|b(\xi)|^{q}\,\dd \xi.$$
	Thus the estimate with $d=1$, $q=2$, $a(x)=(\Re V+\alpha\Im V)_-^{1/2}$ and $b(\xi)=(|\xi|^2+\varepsilon)^{-1/2}$ implies that
	\begin{align*}
		&\sum_{j=1}^{\infty} s_j  ((\Re V+\alpha\Im V)_-^{1/2}(-\Delta+\varepsilon)^{-1/2})^2 \\
		&\leq (2\pi)^{-1} \int_\R |(\Re V(x)+\alpha\Im V(x))_-|\,\dd x\,\int_{\R} (\xi^2+\varepsilon)^{-1}\,\dd \xi.
	\end{align*}
	Note that $\int_{\R} (\xi^2+\varepsilon)^{-1}\,\dd \xi=\pi \varepsilon^{-1/2}$. This proves the claim in the case $d=1$, $p=1$.
\end{proof}	

\begin{remark}
	The proof method extends to other Fourier multipliers $H_0$  for which suitable Cwikel or Kato--Seiler--Simon estimates are available. In that case one again obtains an eigenvalue counting bound with an 
	$\varepsilon$-dependent constant, where the growth rate in $\varepsilon$ is determined by the decay rate of the Fourier multiplier symbol $u(\xi)$ of $(H_0+\varepsilon)^{-1/2}$.
	In particular, the estimate \eqref{eq:Cwikelapplication} with $q=2p$ holds whenever $2<q<\infty$
	and the multiplier symbol satisfies $u\in L^{q,\infty}(\R^d)$.
\end{remark}

\begin{remark}
Using the estimates in Theorem~\ref{thm:CLR-complex} on the number of discrete eigenvalues for a given parameter $\gamma$, in Theorem \ref{thm:real part LT est} below we prove LT inequalities for all parameters that are larger than~$\gamma$.
On the other hand, once we know that an LT inequality of the form
$$\sum_{\lambda_j\in\sigma_{\rm d}(-\Delta+V)} (\Re\lambda_j+\alpha\Im\lambda_{j})_-^{\gamma}\leq \widetilde C_{d,p}\,\int_{\R^d}(\Re V(x)+\alpha\Im V(x))_-^p\,\dd x$$
holds with a constant $\widetilde C_{d,p}>0$ that is independent of $\alpha$ and $V$, we obtain an estimate on the number of discrete eigenvalues as follows: using the lower bound
$$\sum_{\lambda_j\in\sigma_{\rm d}(-\Delta+V)} (\Re\lambda_j+\alpha\Im\lambda_{j})_-^{\gamma}
\geq \sum_{\lambda_j\in\sigma_{\rm d}(-\Delta+V), \atop \Re\lambda_j+\alpha\Im\lambda_j<-\varepsilon} \varepsilon^{\gamma}=\varepsilon^{\gamma}N(\Re\lambda+\alpha\Im\lambda<-\varepsilon;-\Delta+V),$$
we obtain
$$N(\Re\lambda+\alpha\Im\lambda<-\varepsilon;-\Delta+V) \leq \widetilde C_{d,p} \,\varepsilon^{-\gamma} \int_{\R^d}(\Re V(x)+\alpha\Im V(x))_-^p\,\dd x.$$
In particular, if $C_{d,p}^{\rm \,sharp}$ denotes the sharp constant in Theorem~\ref{thm:CLR-complex} and $\widetilde C_{d,p}^{\rm \,sharp}$ denotes the sharp constant in Theorem~\ref{thm:real part LT est}, the above argument implies that 
$$C_{d,p}^{\rm \,sharp}\leq \widetilde C_{d,p}^{\rm\, sharp}.$$
\end{remark}

In the endpoint case in one dimension, our constant is sharp.
\begin{theorem}
Let $d=1$ and $p=1$. Then the constant $C_{d,p}=1/2$ in Theorem~\ref{thm:CLR-complex} is sharp.
\end{theorem}

\begin{proof}
Let $V_t=-t\delta$ for $t>0$ where $\delta\in L^1(\R)$ is the delta-potential. Since $V_t$ is real-valued, the spectrum is real and the inequality in Theorem~\ref{thm:CLR-complex} is independent of $\alpha$. Thus it suffices to prove optimality for $\alpha=0$. For any $t>0$ the point $\lambda_t=-t^2/4$ is a discrete eigenvalue with corresponding eigenfunction $\psi_t(x)=\exp(-t|x|/2)$; in fact, this is the only discrete eigenvalue. 
In order that $\lambda_t<-\varepsilon$, we require $t>2\varepsilon^{1/2}$.
Then
$$\frac{N(\lambda<-\varepsilon;-\frac{\rm d^2}{{\rm d}x^2}+V_t)}{\varepsilon^{-\gamma}\,\int_{\R}(V_t(x))_- \,{\rm d}x}
= \frac{1}{\varepsilon^{-\gamma} t}=\frac{1}{\varepsilon^{-1/2} t}$$
where we have used that $\gamma=1/2$. Now the claim follows since, in the limit  as $t\to 2\varepsilon^{1/2}$, the right-hand side converges to $1/2=C_{d,p}$. 
\end{proof}

\begin{remark}
For other combinations of $d$ and $p$ the sharpness of the constant $C_{d,p}$  in Theorem~\ref{thm:CLR-complex} is an open problem.
For $d\geq 3$ and $\gamma=0$, the constant $C_{d,d/2}$ is the same as Cwikel's constant in his proof of the CLR inequality, see \cite{cwikel1977weak}. In the selfadjoint case, the sharp constant is unknown as well but it is known that e.g.\ Lieb's proof  \cite{Lieb1976CLR} gives a constant better than Cwikel's constant.
The sharpness of the constants in the non-selfadjoint LT inequalities (see Theorem~\ref{thm:real part LT est} below) is also an open problem. For an overview of the results working towards sharp constants in selfadjoint CLR and LT inequalities, see \cite{frank2022schrodinger}.
It is important to remark that, since here we allow for complex potentials as well, the sharp constants may be larger compared to the selfadjoint case.
\end{remark}

\subsection{LT inequalities for complex potentials}

Next we  prove bounds on eigenvalue power sums (Riesz means of order $\gamma$) for eigenvalues in any half-plane $\{\lambda\in\C\,:\,\Re\lambda +\alpha \Im \lambda<0\}$ with $\alpha\in\R$.
Then, in Lemma~\ref{lem:eigenvalue outside cones}, we obtain LT inequalities outside a sector around the essential spectrum; the involved constant depends on the opening angle and diverges with decreasing angle. Integrating over the angle yields a sum over all eigenvalues that depends on the distance to the essential spectrum, see Theorem~\ref{thm:LT complex-B}.

\begin{theorem}\label{thm:real part LT est}
	Let $p=d/2+\gamma$ with $\gamma>1/2$ for $d=1$ and $\gamma>0$ for $d\geq 2$. 
	Then there exists a constant $\widetilde C_{d,p}>0$ such that for all $V\in L^p(\R^d)$ and $\alpha\in\R$,
	$$\sum_{\lambda_j\in\sigma_{\rm d}(-\Delta+V)} (\Re\lambda_j+\alpha\Im\lambda_{j})_-^{\gamma}\leq \widetilde C_{d,p}\,\int_{\R^d}(\Re V(x)+\alpha\Im V(x))_-^p\,\dd x.$$
The constant can be taken as
\begin{equation}\label{eq:consttilde}
\widetilde C_{d,p}= C_{d,p'}^{\,\rm sharp}\,\frac{\gamma^{\gamma+1}}{(\gamma')^{\gamma'}(\gamma-\gamma')^{\gamma-\gamma'}}\, \frac{\Gamma(d/2+\gamma'+1)\Gamma(\gamma-\gamma')}{\Gamma(d/2+\gamma+1)}
\end{equation}
where $C_{d,p'}^{\,\rm sharp}$ is the sharp constant in Theorem~\ref{thm:CLR-complex} for $p'=d/2+\gamma'$ with 
\begin{equation}\label{eq:gamma'}
\gamma'\in [1/2,\gamma)\quad (d=1), \quad \gamma'\in (0,\gamma)\quad  (d=2), \quad \gamma'\in [0,\gamma)\quad (d\geq 3);
\end{equation}
we employ the convention that $(\gamma')^{\gamma'}=1$ if $\gamma'=0$.
\end{theorem}

\begin{proof}
	Let $V\in L^p(\R^d)$.
	We take a sequence of potentials $V_{n}\in C^{\infty}_{c}(\R^{d})$ from the space of infinitely differentiable functions with compact supports such that $\|V_{n}-V\|_{L^{p}(\R^d)}\to 0$ as $n\to\infty$. To prove the estimate in the statement, we restrict ourselves to any arbitrary finite collection of $N$ eigenvalues $\lambda_{j}\in\sigma_{\rm d}(-\Delta+V)$. Similarly as in the proof of~\cite[Lem. 5.4]{han_11}, for $a>0$ sufficiently large we obtain the operator norm convergence
	\[
	\|(-\Delta+V_{n}+a)^{-1}-(-\Delta+V+a)^{-1}\|\to 0
	\]
	as $n\to\infty$. By~\cite[Sect.~IV.3.5]{kato1976perturbation}, this implies that each eigenvalue $\lambda_{j}\in\sigma_{\rm d}(-\Delta+V)$ is the limit of a sequence of eigenvalues $\lambda_{j}(n)\in\sigma_{\rm d}(-\Delta+V_{n})$ as $n\to\infty$, with preserved algebraic multiplicity.
	This convergence argument proves that, without loss of generality, we can assume that $V\in C_c^{\infty}(\R^d)$.
	
	Let $\alpha\in\R$ and $\varepsilon>0$.
	Define $\varepsilon'=(1+t)\varepsilon$ for some $t>0$, and let $p'=d/2+\gamma'$ where 
	$\gamma'$ is as in \eqref{eq:gamma'}.
	Then $V\in L^{p'}(\R^d)$.
	First note that
	$$N( \Re\lambda+\alpha\Im\lambda<-\varepsilon'; -\Delta+V)
	=N(\Re\lambda+\alpha\Im\lambda<-t\varepsilon; -\Delta+V+\varepsilon).$$
	Unfortunately, we cannot directly use Theorem~\ref{thm:CLR-complex} to estimate
	\[
		N(\Re\lambda+\alpha\Im\lambda<-t\varepsilon; -\Delta+V+\varepsilon)
	\] 
	since $V+\varepsilon\notin L^{p'}(\R^d)$. Instead, we  proceed analogously as in the proof of Theorem~\ref{thm:CLR-complex}, using Proposition~\ref{prop:algebraic bd} with $K=\varepsilon$ instead of Theorem~\ref{thm:B-S}, to arrive at the analogous bound in terms of $(\Re V+\alpha\Im V+\varepsilon)_-\in L^{p'}(\R^d)$. Hence
	\begin{align*}
		N( \Re\lambda+\alpha\Im\lambda<-\varepsilon'; -\Delta+V)
		&=N(\Re\lambda+\alpha\Im\lambda<-t\varepsilon; -\Delta+V+\varepsilon) \\
		&\leq  C_{d,p'}^{\,\rm sharp} \, (t\varepsilon)^{-\gamma'}\int_{\R^d}(\Re V(x)+\alpha\Im V(x)+\varepsilon)_-^{p'}\,\dd x\\
		&=C_{d,p'}^{\,\rm sharp}\, t^{-\gamma'}\varepsilon^{-\gamma'}\int_{\R^d}(\Re V(x)+\alpha\Im V(x)+\varepsilon)_-^{d/2+\gamma'}\,\dd x.
	\end{align*}
	Since $\gamma>\gamma'$, an argument due to Aizenman and Lieb \cite[Eq.~(9)]{aizenman1978semi} yields, with substitution $\varepsilon'=(1+t)\varepsilon$,
	\begin{align*}
		&\sum_{\lambda_j\in\sigma_{\rm d}(-\Delta+V)} (\Re\lambda_j+\alpha\Im\lambda_{j})_-^{\gamma}
		=\gamma  \int_0^\infty N(\Re\lambda+\alpha\Im\lambda<-\varepsilon'; -\Delta+V)(\varepsilon')^{\gamma-1}\,\dd\varepsilon'\\
		&\leq   C_{d,p'}^{\,\rm sharp}\,t^{-\gamma'}(1+t)^{\gamma}\, \gamma \int_0^\infty \varepsilon^{\gamma-\gamma'-1}\int_{\R^d}(\Re V(x)+\alpha\Im V(x)+\varepsilon)_-^{d/2+\gamma'}\,\dd x\,\dd \varepsilon\\
		&= C_{d,p'}^{\,\rm sharp}\,t^{-\gamma'}(1+t)^{\gamma} \,\gamma\, \frac{\Gamma(d/2+\gamma'+1)\Gamma(\gamma-\gamma')}{\Gamma(d/2+\gamma+1)}
		\int_{\R^d}(\Re V(x)+\alpha\Im V(x))_-^{d/2+\gamma}\,\dd x.
	\end{align*}
	An easy computation shows that the infimum over $t>0$ is attained at  $t=\frac{\gamma'}{\gamma - \gamma'}$ with
	$$\inf_{t>0}\, t^{-\gamma'}(1+t)^{\gamma}
=
\frac{\gamma^{\gamma}}{(\gamma')^{\gamma'}(\gamma-\gamma')^{\gamma-\gamma'}}.$$
	This implies the claim. 
\end{proof}  

\begin{remark}
	Comparing the assumptions of Theorem~\ref{thm:real part LT est} with the most general assumptions on $\gamma$ in \eqref{eq:gamma}, we see that the result is not applicable in $d=1$ if $\gamma=1/2$, i.e.\ for $p=1$. 
	It remains an open problem whether the result continues to hold in this case.
	The same applies to Theorem~\ref{thm:LT complex-B} below which is a consequence of  Theorem~\ref{thm:real part LT est}.
\end{remark}

\begin{remark}
For $\gamma\geq 1$ it was shown in  \cite{frank2006lieb} that an inequality as in Theorem~\ref{thm:real part LT est} holds and the constant $\widetilde C_{d,p}$ can be taken as the sharp \emph{selfadjoint} LT constant, which is therefore also the sharp constant here. For $\gamma<1$ the sharp constant is unknown; whether it coincides with our $\widetilde C_{d,p}$ in~\eqref{eq:consttilde} is left as an open problem. 
For $d\geq 3$, if we calculate $\widetilde C_{d,p}$ in~\eqref{eq:consttilde} with $\gamma'=0$ and use that $\gamma \,\Gamma(\gamma)=\Gamma(\gamma+1)$, we arrive at
$$\widetilde C_{d,p}=
C_{d,d/2}^{\,\rm sharp} \,\frac{\Gamma(d/2+1)\,\Gamma(\gamma+1)}{\Gamma(d/2+\gamma+1)}.
$$
Note that, with $\Gamma(1)=1$, we obtain $\widetilde C_{d,p}\to C_{d,d/2}^{\,\rm sharp}$ as $p\to d/2$, i.e.\ we recover the sharp CLR constant $C_{d,d/2}^{\,\rm sharp}$ from Theorem~\ref{thm:CLR-complex}.
Furthermore, we make the observation that $\widetilde C_{d,p}=C_d \,L_{\gamma,d}^{\mathrm{cl}} $ where
$$C_d:=
C_{d,d/2}^{\,\rm sharp}\, \Gamma(d/2+1) (4\pi)^{d/2},
\quad
L_{\gamma,d}^{\mathrm{cl}} 
:= (4\pi)^{-d/2} \, \frac{\Gamma(\gamma+1)}{\Gamma(d/2+\gamma+1)}.$$
The \emph{semiclassical LT constant} $L_{\gamma,d}^{\mathrm{cl}}$ is known to be a lower bound for the sharp selfadjoint LT constant, and hence for the sharp constant $\widetilde C_{d,p}^{\,\rm sharp}$ in Theorem~\ref{thm:real part LT est}. This implies the two-sided bound
$$L_{\gamma,d}^{\mathrm{cl}}  \leq\widetilde C_{d,p}^{\,\rm sharp} \leq C_d\, L_{\gamma,d}^{\mathrm{cl}}.$$
\end{remark}

If we take all discrete eigenvalues into account, we arrive at the following LT type inquality.
This is a generalisation of \cite[Thm.~2.1]{boegli2023improved} which was proved only for $\gamma\geq 1$.

\begin{theorem}\label{thm:LT complex-B}
	Let $p=d/2+\gamma$ with $\gamma>1/2$ if $d=1$, $\gamma>0$ if $d=2$ and $\gamma\geq 0$ if $d\geq 3$. Given a continuous, non-increasing function $f:[0,\infty)\to(0,\infty)$, if $f$ satisfies
	\begin{equation}\label{eq:Optimal condition}
		\int_{0}^{\infty} f(x)\;\dd x <\infty,
	\end{equation}
	then there exists a constant $C_{d,p,f}>0$ such that for all $V\in L^{p}(\R^{d})$,
	\begin{equation}\label{eq:LT complex-B}
		\begin{aligned}
			\sum_{\lambda\in\sigma_{\rm d}(-\Delta+V)} \frac{\dist(\lambda,[0,\infty))^{p}}{|\lambda|^{d/2}}
			f\left(-\log\left(\frac{\dist(\lambda,[0,\infty))}{|\lambda|}\right)\right)
			\leq C_{d,p,f} \int_{\R^d} |V(x)|^{p}\;\dd x,
		\end{aligned}
	\end{equation}
	where $C_{d,p,f}=C_{d,p}\left(f(0)+\int_{0}^{\infty} f(x)\;\dd x\right)$ for an $f$-independent constant $C_{d,p}>0$.
\end{theorem}

\begin{remark}
	For any $0<t<1$, Demuth, Hansmann and Katriel proved that there is $C_{d,p,t}>0$ such that for all $V\in L^{p}(\R^{d})$ with $p=d/2+\gamma$ and $\gamma\geq 1$,  
	\[
	\sum_{\lambda\in\sigma_{\rm d}(-\Delta+V)} \frac{\dist(\lambda,[0,\infty))^{p+t}}{|\lambda|^{d/2+t}}
	\leq C_{d,p,t}\int_{\R^d} |V(x)|^{p}\;\dd x.
	\]
	In view of Theorem~\ref{thm:LT complex-B}, we can not only recover the Demuth-Hansmann-Katriel bound by inserting $f(x)=\e^{-tx}$, but we can also extend the result to the wider range of $\gamma$.
\end{remark}

	\begin{remark}\label{rem:accumulation rate}
Under the assumptions of Theorem~\ref{thm:LT complex-B}, we can describe the accumulation rate of discrete eigenvalues to any non-zero points of the essential spectrum. If there exists a sequence of discrete eigenvalues $\{\lambda_{j}\}_{j\in\N}$ of $-\Delta+V$ accumulating at a point $x_{0}\in (0,\infty)$, then $|\Im\lambda_j|\to 0$ as $j\to\infty$ so fast that $\sum_{j\in\N} |\Im\lambda_{j}|^{p} f(-\log(|\Im\lambda_j|/x_0))<\infty$.
	\end{remark}

\begin{remark}\label{rem:sharpcond}
	The integrability condition~\eqref{eq:Optimal condition} is sharp in the sense that if it is omitted, then~\eqref{eq:LT complex-B} cannot hold. A counterexample is given by the one-parameter family of purely imaginary potentials $V:=\ii h\chi_{[-1,1]}$ if $d=1$ and $V:=\ii h\chi_{B_{1}(0)}$ if $d\geq 2$ where $h>0$ and $B_{1}(0)$ is the open unit ball in $\R^{d}$, see \cite{bogli2021lieb,bogli2025lieb,BPoptimal,boegli2023improved}.
\end{remark}

The proof of Theorem~\ref{thm:LT complex-B} relies on the following LT inequality outside the $\kappa$-dependent sector $\{\lambda\in\C\,:\,|\Im\lambda |<\kappa \Re\lambda\}$ for $\kappa>0$.
This is a generalisation of  \cite[Thm.~1]{frank2006lieb}  which was proved only for $\gamma\geq 1$.

\begin{lemma}\label{lem:eigenvalue outside cones}
	Let $p=d/2+\gamma$ with $\gamma>1/2$ if $d=1$, $\gamma>0$ if $d=2$ and $\gamma\geq 0$ if $d\geq 3$. Then there exists a constant $C_{d,p}>0$ such that for all $V\in L^p(\R^d)$ and $\kappa>0$,
	\[
	\sum_{\lambda\in\sigma_{\rm d}(-\Delta+V),\, |\Im\lambda|\geq\kappa\Re\lambda} |\lambda|^{\gamma}
	\leq C_{d,p}\left(1+\frac{2}{\kappa}\right)^{p}\int_{\R^d} |V(x)|^{p}\;\dd x.
	\]
\end{lemma}

\begin{proof} The proof is exactly the same as in  \cite[Thm.~1]{frank2006lieb} but instead of  \cite[Lem.~1]{frank2006lieb} we now employ Theorem~\ref{thm:real part LT est} which is applicable to the larger range of $\gamma$.
\end{proof}

\begin{remark}
With decreasing opening angle, if $\kappa\to 0$, the involved constant diverges as $\kappa^{-p}$. This divergence rate is optimal, see  \cite{BPoptimal,boegli2023improved}, which is shown using the same family of potentials as in Remark~\ref{rem:sharpcond}.
\end{remark}

Now we prove the LT type inequalities.

\begin{proof}[Proof of Theorem~\ref{thm:LT complex-B}]
	We proceed in the exact same way as in \cite[Thm.~2.1]{boegli2023improved} but instead of \cite[Thm.~1]{frank2006lieb} we now employ Lemma~\ref{lem:eigenvalue outside cones} which is applicable to the wider range of~$\gamma$.
\end{proof}

\section*{Acknowledgements}
The authors are grateful to Lukas Schimmer for drawing their attention to reference \cite{LTdetailed}.
The PhD of S.~P.\ is funded through a Development and Promotion of Science and Technology (DPST) scholarship of the Royal Thai Government, ref.\ 5304.1/3758.
The authors would like to thank the Isaac Newton Institute for Mathematical Sciences, Cambridge, for support and hospitality during the programme \emph{Geometric spectral theory and applications}, where work on this paper was undertaken. This work was supported by EPSRC grant EP/Z000580/1.

%
%
%

	\bibliographystyle{acm}
	\bibliography{references}
\end{document}